\documentclass[12pt]{amsart}
\usepackage{a4wide}

\usepackage{amsmath}
\usepackage{amsthm}
\usepackage{amsfonts}
\usepackage{stmaryrd}
\usepackage{amssymb}
\usepackage[all]{xy}
\usepackage{color}
\definecolor{darkblue}{rgb}{0,0,0.9}
\definecolor{lightblue}{rgb}{.5,.5,.9}
\usepackage[pdftex,breaklinks,colorlinks,filecolor=blue,linkcolor=blue,citecolor=blue,urlcolor=blue,hypertexnames=true,plainpages=false]{hyperref}

\usepackage{tikz}
\usepackage{etoolbox}
\usepackage{comment}

\newcommand\R{\mathbb{R}}

\newcommand\Q{\mathbb{Q}}
\newcommand\Z{\mathbb{Z}}

\newcommand\del{\partial}

\newcommand{\Aut}{\textup{Aut}}
\newcommand{\Diff}{\textup{Diff}}
\newcommand{\Symp}{\textup{Symp}}

\newcommand{\Coker}{\textup{Coker}}
\newcommand{\Ker}{\textup{Ker}}
\newcommand{\Id}{\textup{Id}}

\newcommand{\ext}{\textup{e}}
\newcommand{\var}{\textup{var}}

\newcommand{\xra}{\xrightarrow}

\newcommand{\twotwo}[4]{\left( \begin{array}{cc} #1 & #2 \\ #3 & #4 \end{array} \right)}

\theoremstyle{plain}
\newtheorem{theorem}{Theorem}[section]
\newtheorem*{theorem*}{Theorem}
\newtheorem{lemma}[theorem]{Lemma}
\newtheorem{corollary}[theorem]{Corollary}

\newtheorem{conjecture}[theorem]{Conjecture}

\theoremstyle{definition}
\newtheorem{definition}[theorem]{Definition}

\theoremstyle{remark}
\newtheorem{remark}[theorem]{Remark}
\newtheorem{question}[theorem]{Question}

\usepackage{color}

\advance \topmargin by -\headheight
\advance \topmargin by -\headsep
     
\evensidemargin \oddsidemargin
\title{Contact open books with flexible pages}

\author{Jonathan Bowden}
\address{Mathematische Fakult\"{a}t, Universit\"{a}t Regensburg, 93053 Regensburg, Germany}
\email{jonathan.bowden@math.uni-regensburg.de}

\author[Crowley]{Diarmuid Crowley}
\address{Department of Mathematics and Statistics, The University of Melbourne,
Parkville, VIC, 3010, Australia}
\email{dcrowley@unimelb.edu.au}

\date{\today}

\begin{document}
\maketitle

\begin{abstract}
We give 
an elementary
topological obstruction for a manifold
$M$ of dimension $2q{+}1 \geq 7$ to admit a contact open book with flexible Weinstein pages and $c_1(\pi_2(M)) = 0$: if the torsion subgroup of 
the $q$-th integral homology group is non-zero, then no such contact open book exists. 
We achieve this by proving that a symplectomorphism of a flexible Weinstein manifold acts trivially on 
integral cohomology. We also produce examples of non-trivial loops of flexible contact structures using related ideas.
\end{abstract}


\section{Introduction} 
\label{s:intro}
%

Contact manifolds admit open books that are compatible with the contact structure. More precisely, Giroux-Mohsen \cite{Groux} showed that one can always find an open book whose page is a \emph{Weinstein} manifold and a contact form whose Reeb flow is transverse to the pages. In general Weinstein manifolds exhibit interesting geometry, however, there is a class of \emph{flexible} Weinstein manifolds \cite{Cieliebak&Eliashberg12} in dimension at least $6$, based on surgery on loose Legendrians as introduced by Murphy\cite{Murphy}, for which an $h$-principle does hold. Specifically, flexible Weinstein manifolds are determined up to deformation by their almost complex deformation type.

In a similar vein, one has a flexible class of contact structures given by those that are overtwisted by Borman-Eliashberg-Murphy \cite{BEM}. Given these formal parallels one might wonder whether every overtwisted contact structure has a supporting open book with flexible pages. Or more generally:

\begin{question} 
Does every almost contact manifold admit a contact open book with flexible pages?
\end{question}
For specific contact manifolds the answer to this question is known to be false due to work of Courte-Massot \cite{Courte&Massot}, who use the existence of such a supporting flexible open book to deduce boundedness properties for the geometry of the corresponding contactomorphism group. 
However, the general question of whether open books with flexible page exist, for example in a given almost contact class, remains open. 

Assuming that the first Chern class of the contact structure vanishes on $2$-spheres,
we provide a simple topological obstruction to the existence of such open books,
which yields a negative answer to the above question in general. 

\begin{theorem}[An Obstruction to Flexible Open Books]\label{thm:main}
Suppose that a contact manifold $(M,\xi)$ of dimension $2q{+}1 \ge 7$ is such that $c_1(\xi)$ vanishes on $2$-spheres;
e.g.\ $c_1(\xi) = 0$.
If $(M, \xi)$ admits a supporting open book with flexible page, then $H_q(M; \Z)$ is torsion free.
\end{theorem}

\noindent Note that the 
conclusion of Theorem \ref{thm:main}
holds automatically for subcritically fillable manifolds $M^{2q+1}$, since these manifolds have open books with trivial monodromy and a straightforward Mayer-Vietoris argument shows that the $q$-th homology is torsion free (see Section \ref{s:intro} for further details in the general case). Hence one can say that contact manifolds with flexible page behave like subcritically fillable manifolds on the level of {\em homology}.

We further remark that the assumption on the first Chern class ensures that symplectic homology has a well-defined grading. In particular, the above applies to \emph{real} projective spaces, or more generally Lens spaces, in dimension $4q{+}3 \ge 7$. It also applies to the connected sum of any contact manifold with such spaces. Our arguments also apply more generally to contact manifolds with open books whose pages have vanishing symplectic homology. This is then a strictly larger class class of open books in view of Murphy-Siegel \cite{Murphy-Siegel}. We thank Oleg Lazarev for this observation.


The key ingredient in proving Theorem \ref{thm:main} is the following result (restated as
Corollary~\ref{cor:Identity}), which places restrictions on the monodromies of open books with flexible pages. 

\begin{theorem}\label{thm:Identity}
Let $(W_{flex},d\lambda)$ be a flexible Weinstein domain with vanishing first Chern class and let $\varphi$ be an exact symplectomorphism supported in the interior of $W_{flex}$. Then the induced action on integral cohomology is trivial.
\end{theorem}
\noindent Again this result applies more generally to any Liouville manifold (such that $c_1$ vanishes on spheres) whose symplectic homology vanishes.

Theorem \ref{thm:main} can be seen as some sort of suspension of the classification results of Lazarev \cite{Laz},  who used symplectic homology to show that the homology of flexible fillings of contact manifolds given as boundaries of a specific flexible Weinstein domain are determined by the contact structure. Indeed using Lazarev's results we obtain the following:

\begin{corollary}[Homology determined by the binding]\label{cor:main}
Suppose that  $(M,\xi),(M,\xi')$ are contact manifolds of dimension $2q{+}1 \ge 7$ such that
$c_1$ vanishes on $2$-spheres; e.g.\ $c_1(\xi) = 0$.  If each contact manifold admits a flexible supporting open book with contactomorphic bindings,
then $H_*(M; \Z) \cong H_*(M'; \Z)$.
\end{corollary}
Finally, any diffeomorphism which is the identity on a neighbourhood of the boundary and preserves the formal symplectic class yields a loop of contact structures on the boundary. Such loops are obtained by applying the $h$-principle to deform the pulled-back Weinstein structure to the initial one and restricting to the boundary.  If this loop is trivial, then the diffeomorphism is isotopic to a sympectomorphism. 
However, using the results above we can deduce that if the action of $f$ on cohomology is non-trivial, then any resulting loop is non-trivial in the space of (flexible) contact structures on $M = \partial W_{flex}$ (cf.\ Theorem \ref{thm:loops}).

\subsection*{Outline of the proof of Theorem \ref{thm:main}:} Let $W^{2q}$ be the page of a smooth open book decomposition of a $(2q{+}1)$-manifold
\[ M = (\del W \times D^2) \cup_{\Id} T_f, \]
where $f \colon W \to W$ is a diffeomorphism which is the identity near $\del W$
and $T_f := (W \times I)/\!\sim$ is the mapping torus of $f$.  We let $\Diff_\del(W)$ denote the group of diffeomorphisms of $W$ which are the identity near the boundary and let $\pi_0 \Diff_\del(W)$ be
the corresponding mapping class group.

Let $DW = W \cup_{\Id} W$ be the trivial double of $W$, so that $DW = \del (W \times I)$ after smoothing corners.
Define the {\em extension} homomorphism
$$ \ext \colon \Diff_\del(W) \to \Diff(\del (W \times I)), 
\quad f \mapsto f \cup \Id_W.$$
We note that the open book structure on $M$ also expresses $M$ as a {\bf twisted double}
$$ M = (W \times I) \cup_{\ext(f)} (W \times I).$$
In low-dimensions, this corresponds to the familiar path from an open book to a Heegaard decomposition.

The key observation is that, if the Weinstein structure on $W$ is {\em flexible} then the vanishing of symplectic homology can be used to show that the action on homology for the gluing map above is trivial in all degrees up to $q$:
\[ e(f)_* = \Id \colon H_q(DW) \to H_q(DW) \]
From this an easy Mayer-Vietoris argument implies that the homology is torsion free in degree $q$. Note that it is not good enough to simply consider the monodromy of the page, one must extend this to the subdomain of the double 
given by attaching \emph{all} critical handles.
This however provides no problems as this subdomain can be given a flexible Weinstein structure extending the one on $W$ and the diffeomorphism extends by the identity too.

\subsection*{Conventions} All (co)homology groups will be taken with $\Z$-coefficients, unless otherwise indicated. 
All manifolds are smooth, compact, connected and oriented with (possibly empty) boundary.

\subsection*{Acknowledgements} We thank Oleg Lazarev, Klaus Niederkr\"uger, Roger Casals, and Sylvain Courte for helpful comments. We are particularly grateful to Otto Van-Koert for his clarifications concerning monodromies of contact open books,  Zhengyi Zhou for explanations about symplectic homology as well as the anonymous referee for their careful reading and perceptive remarks.

\section{Background}
\subsection{Weinstein Manifolds and Flexibility}
A compact symplectic manifold $(W, \omega)$ with boundary is a Liouville domain if there is a vector field $X$ that is outward pointing at the boundary such that 
$$\omega = L_X\omega = d \iota_X\omega.$$
In particular, the symplectic form is exact and as the precise vector field is not essential, we denote such a manifold simply by $(W,d\lambda)$. If the vector field is {\bf gradient-like} for a Morse function (Lyapanov Function) $\varphi: W \to [0,1]$ having $\partial W$ as a regular level, then $W$ is a {\bf Weinstein domain}. The Weinstein condition then implies that all critical points have index at most $q$ \cite[Chapter 11]{Cieliebak&Eliashberg12}. Moreover, any symplectic form determines an almost complex structure that is unique up to homotopy. This forms the formal data underlying a Weinstein structure. Eliashberg proved the following existence $h$-principle:
\begin{theorem}[Eliashberg, \cite{E}]
Let $(W^{2q},\varphi,J)$ be an almost complex manifold with $q \geq 3$ and boundary admitting a Morse function with critical points of index at most $q$ so that the boundary is a regular level set. Then $W$ admits a Weinstein structure in the same formal class, having $\varphi$ as a Lyapunov function.
\end{theorem}
Moreover, there is an important subclass of Weinstein manifolds built by attaching subcritical handles along isotropic spheres and critical handles along loose Legendrians as introduced by Murphy \cite{Murphy}.

\begin{theorem}[Cieliebak-Eliashberg \cite{Cieliebak&Eliashberg12}]
Let $(W^{2q},\varphi,J)$ be an almost complex manifold with boundary admitting a Morse function having critical points of index at most $q$ so that the boundary is a regular level set. Then $W$ admits a flexible Weinstein structure in the same formal class. Moreover, any two flexible Weinstein structures that are formally homotopic, are homotopic through Weinstein structures.
\end{theorem}
We will also need the following relative version of the above:

\begin{theorem}[Cieliebak-Eliashberg \cite{Cieliebak&Eliashberg12}]\label{thm_CE_extend}
Let $(V_{flex}^{2q},d\lambda,\varphi|_V, J|_V)$ 
be a flexible Weinstein subdomain of $(W^{2q},\varphi,J)$, so that $W$ is obtained from $V$ by attaching handles of index at most $q$ and the almost complex structure $J$ extends the almost complex structure on $V$. Then $W$ admits a flexible Weinstein structure in the same formal class extending that on $V$. Moreover, any two flexible extensions that are formally homotopic, are Weinstein homotopic relative $V_{flex}$.
\end{theorem}
\noindent In what follows we will denote flexible Weinstein manifolds $(W_{flex},d\lambda)$, ignoring the specific Morse function.

\subsection{Symplectic Homology} We will only summarize the relevant facts of symplectic homology that we shall need and refer to the paper of Cieliebak-Oancea \cite{CO} for background. Note that for technical reasons Cieliebak-Oancea only consider symplectic homology with field coefficients, which we will take to be the rationals $\mathbb{Q}$.

\smallskip
\noindent
  \textbf{Assumption:} We henceforth require that $c_1$ vanishes on $\pi_2$ for all manifolds.
\smallskip

\noindent To a Liouville domain $(W,d\lambda)$ one can associate an invariant $SH_*(W,d\lambda)$ which is generated by periodic orbits of admissible Hamiltonians $H$ on a Liouville completion of $W$, which are $C^2$-small on the interior and proportional to a gradient function generating an outward pointing vector field near the boundary. In particular, there is a natural subcomplex generated by stationary orbits whose associated homology is just the ordinary Morse cohomology of $W$. The homology of the complex generated by non-stationary orbits is called the positive symplectic homology $SH^+(W,d\lambda)$, and one has a tautological long exact sequence:
$$ \cdots \to H_{Morse}^{n-k}(W;\mathbb{Q}) \to SH_k(W,d\lambda) \to SH^+_k(W,d\lambda) \to H_{Morse}^{n-k+1}(W;\mathbb{Q}) \cdots $$
\subsection*{Induced maps on $SH^*$} The groups $H_{Morse}^*(W;\mathbb{Q})$ and $SH_*(W,d \lambda)$ are natural with respect to the action of the group of {\em exact} symplectomorphisms%
\footnote{That is those that preserve the primitive $\lambda$ up to a compactly supported exact form $df$.} 
\!with support in the interior of $W$ denoted $\textrm{Symp}_\partial(W,d\lambda)$. More precisely, one acts on the defining chain complexes that involve generic choices of Hamiltonian and almost complex structures and these can then be {\em canonically} identified via Conley resp.\ Floer-continuation maps. Note that the Floer continuation maps specialise to the Conley continuation maps on the Morse complex given by stationary periodic orbits (cf.\ \cite{CO} Proposition 2.10), which gives naturality for the maps $H_{Morse}^{n-k}(W;\mathbb{Q}) \to SH_k(W,d\lambda)$. One can further identify Morse and singular cohomology in a natural manner compatible with the action of diffeomorphisms (\cite[Section 4.2]{Schwarz}).

\subsection*{Induced maps on $SH_+^*(W_{flex},d\lambda$)} It was shown by Lazarev \cite{Laz} that boundaries of {\em flexible} Weinstein domains are {\em asymptotically dynamically convex} (ADC). From this one deduces that the positive symplectic homology is independent of the filling (\cite{Laz}, Proposition 1.8). This is achieved via a neck stretching argument which shows that all Floer cylinders involved in the definition of $SH^+_*(W,d\lambda)$ can be assumed to lie in the symplectic end for a cofinal collection of Hamiltonians. The same argument also holds for all cylinders used to define the relevant continuation maps (cf.\ \cite{Laz}, Proposition 3.9). Thus in the case of flexible domains (or more generally ADC contact manifolds), one has a natural induced action of an element in $\textrm{Symp}_\partial(W,d\lambda)$ on $SH^+_*(W,d\lambda)$, which in turn is {\em trivial}. We remark that in general it is unclear whether a compactly supported symplectomorphism 
induces the identity map on $SH^+_*(W,d\lambda)$.

Now for flexible Weinstein domains one has the following vanishing result, which follows from the exact triangle for symplectic homology established by Bourgeois-Ekholm-Eliashberg \cite{BEE} and the vanishing of Legendrian contact homology for flexible Legendrians (cf.\ \cite[Section 5]{Murphy} and \cite[Proposition 4.8]{EES}):

\begin{theorem} \label{thm:W_flex}
If $(W_{flex},d\lambda)$ is a flexible Weinstein domain, then $SH_*(W_{flex},d\lambda) =0$. 
\qed
\end{theorem}

\noindent As an immediate consequence of the discussion above and Theorem~\ref{thm:W_flex} we have
\begin{corollary}\label{cor:Identity}
Let $(W_{flex},d\lambda)$ be a flexible Weinstein domain and let $\varphi$ be an exact symplectomorphism supported in the interior of $W_{flex}$. Then $\varphi^* = \Id$ on $H^*(W_{flex};\mathbb{Q})$. \qed
\end{corollary}
In particular, this puts strong restrictions on what diffeomorphisms can be realized by (exact) symplectomorphisms in that they must belong to the kernel of the natural map
$$\Diff_\del(W_{flex}) \to \textrm{Aut}(H^*\left(W_{flex})\right).$$
In fact given our current state of understanding one might conjecture the following:
\begin{conjecture}\label{conj:Identity}
Any compactly supported symplectomorphism of a flexible Weinstein structure is isotopic/homotopic to the identity.
\end{conjecture}

To garner some intuition for the conjecture,
it is known that flexible Weinstein manifolds admit no lagrangian spheres so that there can be no Dehn-Seidel twists in their symplectomorphism groups.
\subsection{Open Books}
We recall some definitions as well as the fundamental result of Giroux-Mohsen on the existence of supporting open books on contact manifolds
\begin{definition}[Contact Open Book]
An open book supports a contact structure, if there exists a contact form $\alpha$ that is a contact form on the binding, $d\alpha$ is a symplectic form on the interior of the pages and the  orientations on the binding induced by the contact form and the symplectic form agree.
\end{definition}
Moreover, one can always modify the contact form near the binding so that the Reeb flow is periodic close to the boundary of a neighbourhood of the binding and has a specific form near the binding (cf.\ \cite[Section 3]{DGZ}, \cite{VanKoe}). In particular, one can take a {\em smooth} open book whose monodromy is a symplectomorphism of the page with compact support. However, this monodromy is not the first return map of a Reeb flow. One can then apply a further isotopy to assume that the monodromy is exact, i.e.\ preserves the primitive $\lambda$ up to $df$ for some compactly supported function (cf.\ \cite[Chapter 7.3]{Ge}). By slight abuse we call the resulting map the symplectic monodromy of the open book, as its symplectic isotopy class is well defined (cf.\ \cite{Gr,VanKoe}).
 
 A priori, the page of such an open book need only be a Liouville manifold. However, Giroux-Mohsen showed that such supporting open books always exist and the page can taken to be Weinstein. Thus we have the following:
\begin{theorem}[Giroux-Mohsen \cite{Groux}]
Any contact manifold has a supporting open book, whose page is a Weinstein manifold $(W, \lambda)$ and whose symplectic monodromy lies in the group $\textrm{Symp}_\partial(W,d\lambda)$.
\end{theorem}

\subsection{Variation and the homology of open books}
\label{s:intro}
We compute the homology of an open book
in terms of the action of the map induced by
the monodromy on homology groups and in particular the {\em variation} of the monodromy.

\medskip

 \noindent \textbf{The variation of certain map of pairs:} Let $(X, A)$ be a pair and $f \colon (X, A) \to (X, A)$ a 
 map of pairs such that $f|_A = \Id_A$.
 The (homological) {\em variation} of $f$ is the map
 \[ \var(f) \colon H_*(X, A) \to H_*(X), \]
 which sends a class $[c] \in H_*(X, A)$ represented by a relative cycle $c$ to the class $[f_*(c) - c]$;
see, e.g.\ \cite[\S 2.2]{KK}.
As we shall see, the variation $\var(f)$ of the monodromy $f \colon W \to W$ of 
an open book determines the homology of the open book.
\medskip

 \noindent \textbf{Doubles and a homology long exact sequence:} 
 Let $DW = W_0 \cup_\Id W_1$ denote the double of a $2q$-dimensional manifold with boundary.
To avoid confusion we label the copies of $W$ with $0$ and $1$.
Consider the pair $(DW, W_0)$ and  $f \colon W_0 \to W_0$ a diffeomorphism
which is the identity near the boundary. We then have the diffeomorphism 
\[ \ext(f) \colon (DW, W_0) \to (DW, W_0),\]
defined by setting $\ext(f)|_{W_0} = f$ and $\ext(f)|_{W_1} = \Id_{W_1}$.
%
The inclusion $W_1 \to DW$ admits a retraction defined by the composition of the
inclusion $DW \to W_1 \times I$ and the projection $W_1 \times I \to W_1$.
It follows that the homology long exact sequence of the pair $(DW, W_1)$ splits
and so
\begin{equation} \label{eq:H_*(DW)}
 H_i(DW) \cong H_i(W_1) \oplus H_i(DW, W_1) \cong H_i(W_1) \oplus H_i(W_0, \del W_0) 
\cong H_i(W_1) \oplus H^{2q-i}(W_0). 
\end{equation}
Moreover, the isomorphism
\[ H_i(DW, W_1) \cong \Ker \left[ H_i(DW) \to H_i(W_1 \times I) \right] \]
ensures that in the splitting of \eqref{eq:H_*(DW)}, the summand $H_i(DW, W_1)$ is 
identified with the group $\Ker \left[ H_i(DW) \to H_i(W_1 \times I) \right]$.
The maps $\ext(f)$ and $f$ then induce a map of split short exact sequences:
\[\xymatrix{
H_i(W_1) \ar[d]^\Id \ar[r] & 
H_i(W_1) \oplus H_i(W_0, \del W_0) \ar[d]^{\ext(f)_*} \ar[r] &
H_i(W_0, \del W_0) \ar[d]^{f_*} \ar@/_1.1ex/[dll]_(0.75){\var(f)} \\
H_i(W_1) \ar[r] & 
H_i(W_1) \oplus H_i(W_0, \del W_0)\ar[r] &
H_i(W_0, \del W_0)
}\]
With respect to the splitting of \eqref{eq:H_*(DW)}, the induced map $\ext(f)_* \colon H_i(DW) \to H_i(DW)$
can be expressed as a two-by-two matrix, which we denote $A(\ext(f)_*)$.
The following lemma is presumably well-known, but we include its proof for completeness.

\begin{lemma}  \label{lem:var(f)}
%
%
The matrix $A(\ext(f)_*)$ has the form
\[ A(\ext(f)_*) = \twotwo{\Id}{\var(f)}{0}{f_*}, \]
where we identify $H_i(W_0, \del W_0) = H_i(W_1, \del W_1)$ via the identification $W_0 = W_1$, 
so that $\var(f) \colon H_i(W_1, \del W_1) \to H_i(W_1)$.
%
\end{lemma}

\begin{proof}
The components $\Id, 0$, and $f_*$ of $A(\ext(f)_*)$ are clear.  It remains
to determine the final component.
Let $[c] \in H_i(W_0, \del W_0)$ be represented by a relative $i$-cycle $c$.
The corresponding class in $\Ker \left[ H_i(DW) \to H_i(W_1 \times I) \right]$
is represented by the cycle $c_0 - c_1$, where $c_1 \in H_i(W_1, \del W_1)$
is equal to $c_0$ under the identification $W_0 = W_1$.
Now we have 
\[ \ext(f)_*([c_0 - c_1]) =  [f_*(c_0) - c_1] = [f_*(c_0) - f_*(c_1) + f_*(c_1) - c_1] = f_*([c_0 - c_1]) + \var(f)([c_1]),\]
which completes the proof.
\end{proof}

 \noindent \textbf{Computing the homology of a Weinstein open book:} 
 We now 
 assume that $W$ is the page of a contact open book which is in addition Weinstein; in particular, $W$ is a handlebody which is built with handles of index $q$ or less.
Since $W_0$ is homotopy equivalent to a $q$-dimensional $CW$-complex, it follows that
if $i < q$, then $H^{2q-i}(W_0) = 0$ and in this case the map $H_i(W_j) \to H_i(DW)$ is
an isomorphism.  We now describe the role of the variation in computing the homology groups of $M$.
 
 \begin{lemma} \label{lem:H_*(M)}
Let $M$ be a $(2q{+}1)$-dimensional contact open book with page $W$ and monodromy $f \colon W \to W$.
Then
\[ H_i(M) \cong 
\begin{cases}
H_i(W) & i < q \\
\Coker(\var(f)) & i = q \\
\Ker(\var(f)) & i = q{+}1 \\
H^{2q+1-i}(W) & i > q{+}1.
\end{cases}
\] 
\end{lemma}
 
\begin{remark}
Since the identity $\Id \colon W \to W$ has zero variation, Lemma~\ref{lem:H_*(M)} entails that
$\var(f) = 0$ if and only if $H_*(M)$ is isomorphic to the homology of the trivial open book.
\end{remark} 
 
\begin{proof}[Proof of Lemma~\ref{lem:H_*(M)}]
To keep the notation compact, we write $W' = W \times I$ and label the two copies
of $W'$ in $M$ as $W'_0$ and $W'_1$, so that
\[ M = W'_0 \cup_{\,\ext(f)} W'_1, \]
%
%
%
where we have identified $\del W'_0 = \del W'_1 = DW$.
Now we consider the Mayer-Vietoris sequence for $H_*(M)$ coming from the above decomposition. 
Writing $H_i(W_0, \del W_0) = H^{2q-i}(W_0)$ and $H_*(DW) = H_*(W_1) \oplus H^{2q-*}(W_0)$ we have:
\[ \dots \xra{} H_i(W_1) \oplus H^{2q-i}(W_0) \xra{i_{0*} \oplus i_{1*}} H_i(W'_0) \oplus H_i(W'_1) 
\xra{} H_i(M) \xra{} \dots \]
For $i < q$, we see that the inclusion $H_i(W'_j) \to H_i(M)$ is an isomorphism.
When $i = q$, Lemma~\ref{lem:var(f)} entails that
the map $i_{0*} \oplus i_{1*}$ can be expressed via the matrix
\[ A(i_{0*} \oplus i_{1*}) = \twotwo{\Id}{0}{\Id}{\var(f)}, \]
%
and so there is an exact sequence
\[ 0 \to H_{q+1}(M) \to H^q(W_0) \xra{\var(f)} H_q(W_{1}) \to H_q(M) \to 0.\] 
%
%
When $i > q{+}1$, $H_{i}(M) \cong H^{2q+1-i}(M)$ by Poincar\'e duality, 
and this completes the proof of the lemma.
\end{proof}

%
%
%
%

\medskip

\noindent \textbf{A criterion for trivial variation:}  Observe that a handle decomposition of $W$ induces one on the double 
$DW = W_0 \cup_\Id W_1$ by considering the dual handle decomposition on $W_1$ relative to the boundary 
$\partial W_1 = \partial W$.  Now let $W'$ be the union of all handles of 
index $q$ or less in the above handle decomposition of $DW$,
so that $W' = W_0 \cup W''$, where $W'' \subset W_1$ is the union of a collar of $\del W_1$ and the relative $q$-handles.
Observe that $W' \subset DW$ is a regular neighbourhood of a $q$-skeleton of $DW$
and so the inclusion $W' \to DW$ induces a surjection on homology in degree $q$.
Moreover, $e(f)$ preserves $W'$
and so we have the following

\begin{lemma}\label{lem:torsion_free}
Suppose that $\ext(f)|_{W'}$ induces the identity on $H^q(W'; \Q)$,
then the variation of $f$ is trivial and thus $H_q(M)$ is torsion free.
\end{lemma}

\begin{proof}
The natural isomorphism $H^q(W'; \Q) \cong \mathrm{Hom}(H_q(W'; \Q), \Q)$
entails that $\ext(f)|_{W'}$ induces the identity on $H_q(W'; \Q)$.
Since  $H_q(W'; \Q) \cong H_q(W') \otimes_\Z \Q$ and $H_q(W')$ is torsion free,
it follows that $\ext(f)|_{W'}$ induces the identity on $H_q(W')$.
Since the induced map $H_q(W') \to H_q(DW)$ is onto, the commutative diagram,
\[
\xymatrix{
H_q(W') \ar@{->>}[d] \ar[r]^{\Id} & 
H_q(W') \ar@{->>}[d]\\
H_q(DW) \ar[r]^{\ext(f)_*} & H_q(DW),
} \]
shows that $\ext(f)_* \colon H_q(DW) \to H_q(DW)$ is the identity.
By Lemma~\ref{lem:var(f)}, $\var(f) = 0$ and so $H_q(M) \cong H_q(W)$ by Lemma~\ref{lem:H_*(M)}.
\end{proof}

\section{Proofs}
We are now ready to prove the main results from the Introduction.
\begin{proof}[Proof of Theorem \ref{thm:main}]
Suppose that $(M,\xi)$ is supported by a flexible open book with page $(W_{flex},d\lambda)$. Then this yields a description of $M$ as a twisted double with attaching map $e(f)$ that is the identity on a 
{\em neighbourhood} of $W_1 \subseteq DW$. Moreover, we can arrange that $f \in \Symp_\partial(W_{flex},d\lambda)$ is an exact symplectomorphism with compact support in the interior of $W = W_0$. Furthermore, if $c_1(\xi)$ vanishes on $2$-spheres,
then the same holds for the page, 
since the inclusion $W \to M$ is $q$-connected and $q \geq 3$.

Now the submanifold $W' \subset M$ has a trivial normal bundle and so inherits a stable almost complex structure from
the (almost) contact manifold $(M, \xi)$.
Since homotopy classes of stable almost complex structures are equivalent to actual almost complex structures for open manifolds 
(cf.\ e.g.\ \cite[Lemma 2.16]{BCS}), we see that $W'$ admits an almost complex structure extending the complex structure $(W_0,J_0)$ that underlies the flexible Weinstein structure $(W_{flex},d\lambda)$ on the page. In particular, we can extend the flexible structure to $W'$ by Theorem \ref{thm_CE_extend} and since $e(f) = \Id$ on the extra $q$-handles we thus obtain an exact symplectomorphism $f' = e(f)|_{W'}$ of a flexible Weinstein structure on $W'_{flex} \cong W'$.

By Corollary \ref{cor:Identity}, $f'$ induces the identity on $H^q(W'; \Q)$, and so by Lemma \ref{lem:torsion_free}
the group $H_q(M)$ is torsion free.
\end{proof}
\begin{proof}[Proof of Corollary \ref{cor:main}]
By the main result of Lazarev \cite{Laz}, any two flexible Weinstein fillings $W,W'$ of a fixed flexible contact structure on the binding satisfy $H_*(W) \cong H_*(W')$. Then using the fact that the variation of the monodromy is trivial and applying Lemma \ref{lem:H_*(M)} implies that the same holds for the homology of the manifolds themselves.
\end{proof}

\section{Loops of flexible contact structures}
Consider a flexible Weinstein domain $(W_{flex},d\lambda)$. Let $\Diff_\del(W_{flex})$ denote the group of diffeomorphisms with support in the interior of $W_{flex}$ and let  $\Diff_\del(W_{flex},[J])$ denote the subgroup that preserves the homotopy class of almost complex structures $J=J_\omega$ compatible with $\omega = d\lambda$. Equivalently for any $f \in \Diff_\del(W_{flex},[J])$ the form $\omega$ is homotopic to $f^*\omega$ in the space of non-degenerate (but not necessarily exact) $2$-forms. We then have the following:

\begin{theorem}\label{thm:loops}
Let $(W_{flex},d\lambda)$ be a flexible Weinstein domain of dimension $2q \ge 6$ with contact boundary $(M,\xi_{flex})$ and let $f \in \Diff_\del(W_{flex}, [J])$ be a diffeomorphism that acts non-trivially on rational cohomology and preserves $J$ up to homotopy. Then the fundamental group of the space of contact structures isotopic to $\xi_{flex}$ is non-trivial.
\end{theorem}
\begin{proof}
First observe that the group $\Diff_\del(W_{flex},[J])$ acts on the space of flexible Weinstein structures $\mathfrak{Weinstein}(W_{flex},[J])$ in the fixed homotopy class. Moreover, in view of Eliashberg's $h$-principle pulling back under an $f\in \Diff_\del(W_{flex})$ gives a homotopic Weinstein structure. We thus obtain a deformation $\lambda_t$ via Liouville structures from $\lambda_0 = \lambda$ to $\lambda_1 = f^*\lambda$. This deformation is {\em a priori} non-constant along the boundary. We hence obtain a loop $\xi_t = Ker(\lambda_t)$ of contact structures via restriction to the boundary. Note that the associated loop may depend on the choice of deformation and is not necessarily canonical.

If this loop is contractible then we could alter the Weinstein deformation to be constant near the boundary. But then we could deform the initial diffeomorphism $f$ via an isotopy {\em relative to the boundary} to preserve the Weinstein structure. In particular, in view of Corollary \ref{cor:Identity} this would mean that $f$ acts trivially on ordinary cohomology, contradicting our hypotheses.
\end{proof}
There are many examples of diffeomorphisms of flexible Weinstein domains as in Theorem~\ref{thm:loops}.
For example, let $W = \sharp_g(S^q \times S^q) \setminus \mathrm{Int}(D^{2q})$ be a punctured copy of the $g$-fold
connected sum of a product of spheres $S^q \times S^q$ endowed with an almost complex structure $J$ and consider $W_{flex}$
with boundary $(S^{2q-1}, \xi_{flex})$.
If the natural map $\pi_q(SO) \to \pi_q(SO/U)$ vanishes; i.e.\ if $q \not \equiv 0, 7$ mod~$8$, then
$J$ is homotopic to $f^*J$ if and only if the Chern classes $c_{q/2}(W,J)$ and $c_{q/2}(W,f^*J)$ agree. 
This is because $W$ retracts onto a wedge of $q$-spheres
and so the sole obstruction to finding a homotopy between $J$ and $f^*J$ 
lies in the group $H^q(W; \pi_q(SO/U))$ and has image a non-zero multiple 
of $c_{q/2}(W,J) - c_{q/2}(W,f^*J)$ in the torsion free group $H^q(W; \pi_q(BU))$.
Of course, if $c_{q/2}(W, J) = 0$ then $c_{q/2}(W,J) = c_{q/2}(W,f^*J)$ and we note that $W$ always admits $J$
with vanishing Chern classes: if $q$ is odd then $c_{q/2} = 0$ by definition, and if $q$ is even
this follows as $W$ is parallelisable.
Since the map given by the induced action on cohomology, 
\[ \Diff_\del(W) \to \Aut(H^q(W)),\]
is surjective onto the automorphisms of $H^q(W)$ respecting the cup product pairing $\lambda_W$
and a certain quadratic refinement%
\footnote{When $q$ is even or $q = 3, 7$, $\mu_W$ is determined by $\lambda_W$ and so may be ignored.}
$\mu_W$ of $\lambda_W$ \cite[Theorem 2]{K},
we obtain many non-trivial loops of flexibly filled contact structures on $S^{2q-1}$.
If $q \not \equiv 0, 7$~mod~$8$, we can take any automorphism $\alpha \in \Aut((H^q(W), \lambda_W, \mu_W), c_{q/2}(J))$ 
and it will be realized by an element of $\Diff_\del(W)$, which preserves the homotopy class of $J$.
If the automorphism $\alpha$ has infinite order, so does the loop of contact structures.

When $q \equiv 0, 7$ mod~$8$ and $q \neq 7$, results and arguments of Randal-Williams 
\cite[Theorem 3.1 \& the discussion prior to Corollary 4.1]{R-W} prove that  
certain generalised Dehn twists generate
a subgroup of $\Diff_\del(W)$
which acts trivially on $H^*(W)$ and transitively on  
each set of homotopy classes of almost complex structures with $c_{q/2}$ fixed.
This implies that the 
discussion above also holds for these values of $q$.
When $q = 7$, a complex structure $J$ defines a quadratic refinement $\mu_J$ of $\lambda_W$,
and results from \cite{C} and \cite{K} entail that the above holds with $\mu_W$ replaced by $\mu_J$.


\begin{remark}
As pointed out to us by the referee, Theorem \ref{thm:loops} also applies if the (homological) variation of the diffeomorphism is non-trivial, 
even though the action on homology is trivial. 
Such a diffeomorphism is given by the (smooth) 
Dehn-Seidel twist $\tau$ on the flexibilisation of the cotangent disc bundle $W = D(T^*S^q$) of any $q$-sphere for $q$ odd.
\end{remark}

\end{document}